\documentclass[11pt]{amsart}

\usepackage{a4wide,xcolor,eucal,enumerate,mathrsfs}

\usepackage{pdfsync}

\usepackage[hidelinks]{hyperref}

\usepackage{amsmath,amssymb,epsfig,amsthm}
\usepackage[latin1]{inputenc}

\usepackage[english]{babel}
\usepackage{amsmath,amsfonts,amsthm,amssymb,dsfont,url}
\usepackage[babel]{csquotes}

\usepackage[style=alphabetic,backend=bibtex]{biblatex}
\bibliography{bibliografia} 

\theoremstyle{definition}
\newtheorem{de}{Definition}[section]
\newtheorem{ex}[de]{Example}
\theoremstyle{plain}
\newtheorem{teo}[de]{Theorem}

\newtheorem{prop}[de]{Proposition}
\newtheorem{cor}[de]{Corollary}
\theoremstyle{remark}
\newtheorem{oss}[de]{Remark}

\newcommand{\R}{\mathbb{R}}
\newcommand{\p}{\mathbb{P}}

\newcommand{\E}{\mathbb{E}}

\newcommand{\C}{\mathcal{C}}
\newcommand{\F}{\mathcal{F}}

\renewcommand{\epsilon}{\varepsilon}
\newcommand{\veps}{\varepsilon}

\newcommand{\ang}[1]{\left< #1 \right>}

\title[Functional Cramer-Rao bounds in Sobolev spaces]{Functional Cramer-Rao bounds and Stein estimators in Sobolev spaces, for Brownian motion and Cox processes}

\author{Eni Musta }
\address{Delft University of Technology}
\email{e.musta@tudelft.nl}

 \author{Maurizio Pratelli}
   \address{Universit\`a degli Studi di Pisa} 
   \email{pratelli@dm.unipi.it}

   \author{Dario Trevisan}
   \address{Universit\`a degli Studi di Pisa} 
   \email{dario.trevisan@unipi.it}

\thanks{The second and third authors are members of the GNAMPA group of the Istituto  Nazionale di Alta Matematica (INdAM)}

\begin{document}

\begin{abstract}
We investigate the problems of drift estimation for a shifted Brownian motion and intensity estimation for a Cox process on a finite interval $[0,T]$, when the risk is given by the energy functional associated to some fractional Sobolev space $H^1_0\subset W^{\alpha,2}\subset L^2$. In both situations, Cramer-Rao lower bounds are obtained, entailing in particular that no unbiased estimators with finite risk in $H^1_0$ exist. By Malliavin calculus techniques, we also study super-efficient Stein type estimators (in the Gaussian case).
\end{abstract}

\maketitle 
\section{Introduction}

In this paper we focus on two problems of non-parametric (or, more rigorously, infinite-dimensional parametric) statistical estimation: drift estimation for a shifted Brownian motion and intensity estimation for a Cox process, on a finite time interval $[0,T]$. Our investigation stems from the articles \cite{MR2458197, MR2486115} where N.\ Privault and A.\ R{\'e}veillac developed an original approach to these problems, by employing techniques from Malliavin calculus and the so-called \emph{Stein's method} \cite{MR0133191} to study Cramer-Rao bounds and super-efficient ``shrinkage'' estimators in these infinite-dimensional frameworks. Such a combination of these two powerful techniques fits into a more general picture, which only in the recent years has become clear (see the monograph \cite{nourdin-peccati}) and is currently a very active research area, with impact on statistics (see e.g.\ \cite{MR1873834, MR2857018, MR2870511, MR3089887}) and, more generally, on probabilistic approximations.

As in \cite{MR2458197, MR2486115}, we assume that the unknown function to be estimated belongs to the Hilbert space $H^1_0(0,T)$ (which is a reasonable choice, at least in the case of shifted Brownian motion, because of Cameron-Martin and Girsanov theorems) but we move further by addressing the following question, which is rather natural but apparently was not considered: what about estimators which also take values in $H^1_0$? Indeed, in \cite{MR2458197, MR2486115}, estimators are seen as functions with values in $L^2( [0,T], \mu)$ (where $\mu$ is any finite measure) or, equivalently, the associated risk is computed with respect to the $L^2$ norm and not the (stronger) $H^1_0$ norm.

To investigate this problem, we first provide Cramer-Rao bounds with respect to different risks, by considering the estimation in the interpolating fractional Sobolev space $H^1_0\subset W^{\alpha,2}\subset L^2$, for $\alpha \in [0,1]$. It turns out that no unbiased estimator exist in $H^1_0$ (Theorem \ref{teo:cr-h1-bm}) and even in $W^{\alpha,2}$, for $\alpha\geq 1/2$ (Theorem \ref{prop:cramer-rao-wiener-fractional-hilbert}). Although a bit surprising, these results reconcile with the following intuition: since the estimator is a function of the realization of the process, whose paths also do not belong to $H^1_0$ (nor $W^{\alpha,2}$, for $\alpha\geq 1/2$), it is ``too risky'' to estimate (without bias) the parameter in that scale of regularity. Therefore, besides answering a rather natural question, our results highlight the delicate role played by the choice of different norms in such estimation problems, and one might expect that similar phenomena might appear in other situations, technically more demanding (e.g.\ SDE's).

As a second task, we study super-efficient ``shrinkage'' estimators in the spaces $W^{\alpha, 2}$. It is often intuitively suggested the ideal situation for the problem of estimation would be to have an unbiased estimator with low variance, but allowing for a little bias may entail existence of estimators with lower risks, in many situations: this is the purpose of \emph{Stein's method}, and we rely on its extension and combination with Malliavin calculus to these frameworks developed in \cite{MR2458197, MR2486115}. With a similar approach, we give sufficient conditions for super-efficient estimators in $W^{\alpha, 2}$, for $\alpha < 1/2$, and we give explicit examples of such estimators, in the case of Brownian motion (Example \ref{ex:stein-bm}). In the case of Cox processes, although it is possible to define a suitable version of Malliavin calculus and provide as well sufficient conditions for Stein estimators, we are currently unable to provide explicit examples.

The paper is organized as follows. In Section \ref{sec:drift-bm} we deal with drift estimation for a shifted Brownian motion, addressing Cramer-Rao lower bounds with respect to risks computed in $H^1_0$ and fractional Sobolev spaces. Analogous results on intensity estimators for Cox processes are given in Section \ref{sec:cox}. In Section \ref{sec:malliavin}, we recall notation and results for Malliavin calculus on the Wiener space. Finally, in Section \ref{sec:stein}, we discuss super-efficient estimators.

\section{Drift estimation for a shifted Brownian motion}\label{sec:drift-bm}

In this section, we fix $T \ge 0$ and let $X = (X_t)_{t \in [0,T]}$ be a Brownian motion (on the finite interval $[0,T]$), defined on some filtered probability space $(\Omega,\F,(\F_t)_{t \in [0,T]},\p)$. As a (infinite-dimensional) space of parameters $\Theta$, we consider a set of absolutely continuous, adapted processes $u_t :=\int_0^t\dot{u}_s\,ds$ (for $t \in [0,T]$) such that $(\dot{u}_t)_{t \in [0,T]}$ satisfies the conditions of Girsanov theorem: indeed, for $u \in \Theta$, we define the probability $\p^u:=L^u\p$, with
\[
L^u:=\exp\Big[\,\int_0^T \dot{u}_s\,dX_s-\frac{1}{2}\int_0^T \dot{u}^2_s\,ds\,\Big],
\]
and Girsanov theorem entails that, with respect to the probability $\p^u$, the process $X^u_t :=X_t-u_t$ is a Brownian motion on $[0,T]$.

We address the problem of estimating the drift w.r.t.\ $\p^u$ on the basis of a single observation of $X$. This is of interest in different fields of applications: for example, we can interpret $X$ as the observed output signal of some unknown input signal $u$, perturbed by a Brownian noise. Such a problem is investigated e.g.\ in \cite{MR2458197}, where the following definition is given.  

\begin{de}
Any measurable stochastic process $\xi:\Omega\times[0,T]\to \R$ is called an estimator of the drift $u$. An estimator of the drift $u$ is said to be unbiased if, for every $u \in \Theta$, $t \in [0,T]$, $\xi_t$ is $\p^u$-integrable and it holds $\E^u[\,\xi_t\,]=\E^u[\,u_t\,]$.
\end{de}

In this section, we forgo to specify ``of the drift $u$'' and we simply refer to estimators. Moreover, we refer to the quantity $\E^u[\,\xi_t-u_t\,]$ as the bias of the estimator $\xi$ (whenever it is well-defined).

By introducing as a risk associated to any estimator $\xi$, the quantity
\begin{equation}\label{eq:risk-l2}
\E^u[\,\Vert \xi-u\Vert^2_{L^2(\mu)}\,]=\E^u\Big[\,\int_0^T \vert \xi_t-u_t\vert^2\,\mu(dt)\,\Big],
\end{equation}
where $\mu$ is any finite Borel measure on $[0,T]$, Privault and R\'{e}veillac  provide the following Cramer-Rao lower bound for adapted and unbiased estimators \cite[Proposition 2.1]{MR2458197}, $\Theta$ being the space of all absolutely continuous, adapted processes, whose derivatives satisfy the conditions of Girsanov theorem.

\begin{teo}[Cramer-Rao inequality in $L^2(\mu)$] For any adapted and unbiased estimator $\xi$ it holds
\begin{equation}\label{eq:cramer-rao-pr}
\E^u[\,\Vert \xi-u\Vert^2_{L^2(\mu)}\,] \geq \int_0^T t\,\mu(dt), \quad \text{for every $u \in \Theta$.}
\end{equation}
Equality is attained by the (efficient) estimator $\hat{u}=X.$
\end{teo}

Before giving our results, let us briefly comment on some aspects of this inequality and its proof, in particular with respect to adaptedness of $\xi$ and the role played by the exponent $2$.

By direct inspection of the proof in \cite{MR2458197}, the requirement for $\xi$ to be adapted is seen to be unnecessary. Indeed, the argument relies on an application of Cauchy-Schwarz inequality in the right hand side of the identity
\begin{equation}
\label{eq:identity-cramer-rao-bm}
v(t) = \E^u\Big[\,(\xi_t-u_t)\,\int_0^T\dot{v}(s)\,dX^u_s\,\Big], \quad \text{for $t \in [0,T]$,}
\end{equation}
valid for every \emph{deterministic} process $v \in \Theta$ (thus, $v(t) := \int_0^t \dot{v}(s)\,ds$) and then choosing $\dot{v}(s) = 1_{[0,t]}(s)$. In turn, the proof of \eqref{eq:identity-cramer-rao-bm} uses fact that, for every $\veps \in \R$, it holds $u+\epsilon\,v \in \Theta$, thus
\[
\E^{u+\epsilon\,v}[\,\xi_t\,] =\E^{u+\epsilon\,v}[\,u_t+\epsilon\,v(t)\,]
 =\E^{u+\epsilon\,v}[\,u_t\,]+\epsilon\,v(t), \quad \text{ for $t\in[0,T]$.}
\]
and differentiates with respect to $\veps$ at $\epsilon =0$ (exchanging between differentiation and expectation is justified by the finitness of the left hand side in \eqref{eq:cramer-rao-pr}, otherwise there is nothing to prove):
\begin{equation*}\begin{split}
\frac{d}{d\epsilon} \Big|_{\epsilon=0}\E^{u+\epsilon\,v}[\,\xi_t-u_t\,] &= \E\Big[\,(\xi_t-u_t)\,\frac{d}{d\epsilon}\Big|_{\epsilon=0} L_T^{u+\epsilon\,v}\,\Big] \\ & = 
 \E^u\Big[\,(\xi_t-u_t)\,\int_0^T\dot{v}(s)\,dX^u_s\,\Big].
\end{split}
\end{equation*}

Let us also notice that it is not necessary for $\Theta$ to be the whole set of drifts $u$ such that Girsanov theorem applies to $\dot{u}$, and the following condition is sufficient: for every $u \in \Theta$ and deterministic $v \in \Theta$, it holds $u+ v \in \Theta$.

\begin{oss}\label{oss:adapted}
Back to the problem of adaptedness of $\xi$, it would be desirable to argue that general (not-necessarily adapted) estimators can not perform better than adapted ones, and the following argument might seem to go in that direction, but does not allow us to conclude. Let $\xi$ be any unbiased estimator and for $u \in \Theta$, consider the optional projection $\eta$ of $\xi$, with respect to the probability $\p^u$, so that $\eta_t:=\E^u[\,\xi_t\,|\,\F_t\,]$, for $t \in [0,T]$. Then, $\E^u [\,\eta_t\,] = u_t$ and it holds
\[
\E^u[\,|\eta_t-u_t|^2\,]=\E^u\big[\,\E^u[\xi_t-u_t\,|\,\F_t\,]^2\,\big]\leq \E^u[\,|\,\xi_t-u_t\,|^2\,].
\]
However, this does not entail that $\eta$ performs better that $\xi$, since $\eta = \eta^u$ depends also on $u$, thus it is not an estimator. On the other side, if we keep $\bar{u} \in \Theta$ fixed, then $\eta^{\bar u} $ could be biased, i.e.\ $\E^u[\,\eta_t^{\bar{u}}\,]\neq \E^u[\,u_t\,]$ for some $u \in \Theta$, $t \in [0,T]$.
\end{oss}

\begin{oss}
\label{oss:lp}
Similarly to the mean squared error, one can consider the risk defined by $L^p$ norms, for $p \in (1,\infty)$:
\[
\int_0^T\E^u[\,|\,\xi_t-u_t\,|^p\,]\,\mu(dt).
\]
Again, by direct inspection of the proof in  \cite{MR2458197}, applying H\"{o}lder inequality (with conjugate exponents $(p,q)$) instead of Cauchy-Schwarz inequality in \eqref{eq:identity-cramer-rao-bm}, we obtain an inequality of the form
\[
\E^u[\,|\xi_t-u_t|^p\,]\geq \frac{|v(t)|^p}{c_q^{p/q}\,\Big(\int_0^t\dot{v}^2(s)\,ds\Big)^{p/2}}\geq \frac{1}{c_q^{p/q}}\,t^{p/2}, \quad \text{for $t \in [0,T]$,}
\]
where $c_q:=\E[\,|Y|^q\,]$ is the $q$-th moment of a $N(0,1)$ random variable $Y$. Integration with respect to $\mu$ then provides a Cramer-Rao type lower bound. However, letting $\xi = X$, one has
\[
\E^u[\,|X_t-u_t|^p\,]=\E^u[\,|X^u_t|^p\,]=c_p\,t^{p/2},\quad \text{for $t \in [0,T]$,}
\]
thus $X$ is not an efficient estimator in $L^p(\Omega\times[0,T])$ for $p\neq 2$.
\end{oss}

In all what follows, we let $H^1_0 (= H^1_0(0,T))$ be the space of (continuous) functions in the form $h(t)=\int_0^t\dot{h}(s)\,ds$, for $t \in [0,T]$, with $\dot{h}\in L^2(0,T)$ (usually called, in this context, the Cameron-Martin space), and we assume that, for every $u \in \Theta$, $h \in H^1_0$, it holds $u+h \in \Theta$. The $H^1_0$ ``energy'' functional, namely $\|h\|_{H^1_0} := \|\dot{h}\|_{L^2(0,T)}$ provides a Hilbert norm on $H^1_0$. For simplicity of notation, we extend such a functional identically to $+\infty$ for any Borel curve $h: [0,T] \to \R$ which do not belong to $H^1_0$.

We notice that $H^1_0$ is included in $\C^{1/2}(0,T)$, the space of $1/2$-H\"{o}lder continuous functions: since the paths of the Brownian motion are not in $1/2$-H\"{o}lder continuous, we deduce that the process $X$ is not $H^1_0$-valued (negligibility of the Cameron-Martin space holds true also for abstract, infinite-dimensional, Wiener spaces). However, since the drift $u$ takes values in $H^1_0$, it is natural to look for an estimator $\xi$ sharing this property. Our first result shows that, if we require $\xi$ to be unbiased, this is not possible, i.e.\ such an estimator $\xi$ has necessarily infinite $H^1_0$ risk.

\begin{teo}[Estimators in $H^1_0$]
\label{teo:cr-h1-bm}
Let $\xi$ be an estimator such that, for some $u \in \Theta$, it holds
\[
\E^u[\,\| \xi-u\|^2_{H^1_0} \,]<\infty.
\] 
Then, $\xi$ is not unbiased.
\end{teo}

Before we address the proof for general, possibly non-adapted, estimators, we give the following argument that exploits Ito formula: actually it is longer, but we feel that it is more of stochastic flavor.

\begin{proof}(Case of adapted estimators.)
Let us assume, by contradiction, that $\xi$ is unbiased, thus by difference, $\xi \in L^2( \Omega, \p^u; H^1_0)$. For every (deterministic) $v \in H^1_0$, arguing as above for the deduction of \eqref{eq:identity-cramer-rao-bm}, we obtain that
\[
v(t)=\E^u\Big[\,\int_0^t(\dot{\xi}_s-\dot{u}_s)\,ds\,\int_0^t\dot{v}(s)\,dX^u_s\,\Big], \quad \text{for $t \in [0,T]$,}
\]
where stochastic integration reduces to the interval $[0,t]$ because of the adaptedness assumption. Integrating by parts (i.e., using Ito's formula) we rewrite the random variable above as
\[
\int_0^t\Big(\int_0^s\dot{v}(r)\,dX^u_r\Big)\,(\dot{\xi}_s-\dot{u}_s)\,ds+\int_0^t\Big(\int_0^s(\dot{\xi}_r-\dot{u}_r)\,dr\Big)\,\dot{v}(s)\,dX^u_s
\]
obtaining the right analogue of \eqref{eq:identity-cramer-rao-bm} for the study of $H^1_0$ energy:
\[
v(t)=\E^u\Big[\,\int_0^t\Big(\int_0^s\dot{v}(r)\,dX^u_r\Big)\,(\dot{\xi}_s-\dot{u}_s)\,ds \,\Big], \quad \text{for $t \in [0,T]$.}
\]
Indeed, Cauchy-Schwarz inequality and Ito's isometry give
\begin{align*}
v(t)^2&\leq \E^u\Big[\,\int_0^t \big(\int_0^s\dot{v}(r)\,dX^u_r\big)^2\,ds \,\Big]\,\E^u\Big[\,\int_0^t (\dot{\xi}_s-\dot{u}_s)^2\,ds \,\Big]\\
&=\int_0^t\Big(\int_0^s\dot{v}^2(r)\,dr\Big)\,ds \,\int_0^t\E^u[\,(\dot{\xi}_s-\dot{u}_s)^2\,]\,ds\\
&= \int_0^t (t-s) \dot{v}^2(s) ds \int_0^t\E^u[\,(\dot{\xi}_s-\dot{u}_s)^2\,]\,ds.
\end{align*}
In particular, choosing $t = T$, we deduce
\[
\E^u[\,\| \xi-u\|^2_{H^1_0} \,] \geq \frac{v(T)^2}{\int_0^T(T-t) \dot{v}^2(t)\,dt}.
\]
To obtain a contradiction, it is enough to prove that for every constant $c>0$, there exists $\dot{v}\in L^2(0,T)$ such that the left hand side above is greater than $c$, i.e.,
\begin{equation}
\label{eqn:inf}
\left(\int_0^T \dot{v}(t)\,dt\right)^2\geq c\,\int_0^T(T-t)\,\dot{v}(t)^2\,dt.
\end{equation}
Indeed, if we let $\dot{v}(t)=\frac{1}{(T-t)^\alpha}$ for some $0<\alpha<1$, it holds
\[
\left(\int_0^T \dot{v}(t)\,dt\right)^2=\left(\frac{T^{1-\alpha}}{1-\alpha}\right)^2\quad\text{and}\quad \int_0^T(T-t)\,\dot{v}^2(t)\,dt=\frac{T^{2(1-\alpha)}}{2(1-\alpha)}.
\]
It is then sufficient to let $\alpha \uparrow 1$ to conclude. 
\end{proof}

\begin{oss}
Instead of the explicit construction of $v \in H^1_0$ above, to obtain a contradiction we can also use the following duality result. On a measure space $(E,\mathcal{E},\mu)$, if $g \ge 0$ is a measurable function such that, for some constant $c>0$, it holds 
\[
\int_E f\,g\,d\mu\leq c\,\left(\int_E f^2\,d\mu\right)^{1/2},\quad \text{for every $f\in L^\infty(\mu)$, $f\geq 0$,}
\]
then it holds $g\in L^2(\mu)$ with $\|
 g\|_{L^2(\mu)}\le c$. The easy proof follows from considering the continuous, linear functional $\phi$ initially defined on $L^\infty\cap L^2(\mu)$ by $f \mapsto \int_E f\,g\,d\mu$ and then apply Riesz theorem on its extension to $L^2(\mu)$.

In the proof above, a contradiction immediately follows from \eqref{eqn:inf}, letting $\mu(dt)=(T-t)\,dt$ and $g(t)=(T-t)^{-1}$.
\end{oss}

We now provide a complete proof of Theorem \ref{teo:cr-h1-bm}.

\begin{proof}(General case.)
Arguing by contradiction, we let $\xi\in L^2( \Omega, \p^u; H^1_0)$. For every (deterministic) $v \in H^1_0$, arguing as above for the deduction of \eqref{eq:identity-cramer-rao-bm}, we obtain instead
\[v(t)=\E^u\Big[\,\int_0^t(\dot{\xi}_s-\dot{u}_s)\,ds\,\int_0^T\dot{v}(s)\,dX^u_s\,\Big], \quad \text{for $t \in [0,T]$.}\]
Then, we differentiate with respect to $t \in [0,T]$ (exchanging derivatives and expectation is ensured by the finite risk assumption), and we obtain, for a.e.\ $t \in [0,T]$, 
\[ \dot{v}(t) = \E^u\Big[\, (\dot{\xi}_t-\dot{u}_t) \,\int_0^T\dot{v}(s)\,dX^u_s\,\Big],\]
At this stage, Cauchy-Schwarz inequality and Ito isometry yield
\begin{equation}\label{eq:inequality-h1-useful-crbm}|\dot{v}(t)|^2 \le  \E^u\Big[\, |\dot{\xi}_t-\dot{u}_t|^2 \,\,\Big]  \int_0^T |\dot{v}(s)|^2 ds,\quad \text{for a.e.\ $t \in [0,T]$,}\end{equation}
From this inequality, we easily obtain a contradiction, arguing as follows. Let $A \subseteq [0,T]$ be a non-negligible Borel subset such that $\int_A \E^u[ |\dot{\xi}_t-\dot{u}_t|^2]dt <1$, which exists because of the finite risk assumption and uniform integrability (notice that $A$ does not depend upon $v$). Then, integrating the above inequality for $t \in A$, we obtain
\[ \int_A |\dot{v}(t)|^2 dt \le \int_A \E^u\Big[\, |\dot{\xi}_t-\dot{u}_t|^2 \,\,\Big] dt \int_0^T |\dot{v}(t)|^2 dt,\]
for every $\dot{v} \in L^2(0,T)$, in particular for every $\dot{v} \in L^2( A)$. Simply taking $\dot{v} = 1_A$, we obtain the required contradiction.
\end{proof}

Actually, the result on the absence of unbiased estimators in $H^1_0$ can be slightly strengthened, allowing for estimator whose bias is sufficiently regular. We state it as a corollary (of the proof), remarking that similar deductions could be performed also in the cases that we consider below.

\begin{cor}
Let $\xi$ be an estimator such that, for every $u \in \Theta$, $t \in [0,T]$, $\xi_t$ is $\p^u$-integrable, and it holds, for some $C = (C_t)_{t\in [0,T]} \in L^2(0,T)$ (possibly depending upon $u \in \Theta$),
\[
\left| \frac{d}{dt} \frac{d}{d\epsilon}\Big|_{\epsilon=0} \E_{u+ \veps v }[\,\xi_t-u_t\,] \right | \leq C_t \|v\|_{H^1_0},\quad\text{a.e.\ $t \in [0,T]$, for every $v \in H^1_0$.}
\]
Then, the $H^1_0$ risk of the estimator $\xi$ is infinite, i.e.\
\[
\E^u[\, \| \xi - u \|^2_{H^1_0} \,]\,ds=\infty, \quad \text{ for every $u \in \Theta$.}
\] 
\end{cor}
\begin{proof}
We argue exactly as in the proof above, but we write
\[
\E^{u+\epsilon\,v}[\,\xi_t\,]=\E^{u+\epsilon\,v}[\,u_t\,]+\epsilon\,v(t)+b_t^{u+\epsilon\,v}.
\]
where $b_t^u := \E^u[\,\xi_t-u_t\,]$ is the bias. After differentiation with respect to $\epsilon$ and $t$, we obtain \eqref{eq:inequality-h1-useful-crbm} with $\E^u[ |\dot{\xi}_t-\dot{u}_t|^2]+ C_t^2$ in place of $\E^u[ |\dot{\xi}_t-\dot{u}_t|^2]$ and we conclude arguing as in the proof above.
\end{proof}

We address now analogous results for the intermediate spaces $H^1_0\subset W^{\alpha,2}\subset L^2$, for $\alpha\in(0,1)$, defined as follows.
 
\begin{de}
For $\alpha\in (0,1)$, $p\in(1,\infty)$, the fractional Sobolev space $W^{\alpha,p} ( = W^{\alpha, p}(0,T))$ is defined as the space of functions $u \in L^p(0,T)$ such that their ``energy'' functional
\[
\|  u \|^p_{W^{\alpha,p}_0}:= \int_0^T\int_0^T\frac{|u_t-u_s|^p}{|t-s|^{p \alpha +1}}\,dt\,ds
\]
is finite.
\end{de}

We refer to \cite{MR2944369} for a survey of the theory of fractional Sobolev spaces, although here we need nothing more than the definition above. The space $W^{\alpha,p}$, endowed with a suitable norm, interpolates (in the sense that could be made precise) between the Sobolev space $W^{1,p}$ and $L^p$; for example, it holds $W^{\alpha',p}\subseteq W^{\alpha,p}$ for $0< \alpha\leq\alpha'<1$, and $W^{\alpha, 2} \subseteq H^1$, with
\begin{equation}\label{eq:walpha-h1} \|  u \|^2_{W^{\alpha,2}_0} \le 2 \int_0^T |\dot{u}_r|^2  \int_r^T\int_0^r\frac{1}{|t-s|^{2 \alpha}}\,ds\,dt\,dr \le C_{\alpha,T} \| u\|_{H^1_0}^2.\end{equation}
From this inequality, the above theorem for estimators in $H^1_0$ could be also obtained by the next results.

Let us first consider the Cramer-Rao bound in the quadratic case.
\begin{teo}[Cramer-Rao inequality in $W^{\alpha,2}$] \label{prop:cramer-rao-wiener-fractional-hilbert} Let $\xi$ be an unbiased estimator. For every $\alpha\in(0,1)$, it holds
\[
\E^u\left[\Vert \xi-u\Vert^2_{W^{\alpha,2}_0}\right]\geq \int_0^T\int_0^T\frac{1}{|t-s|^{2\alpha}}\,dt\,ds, \quad \text{for every $u \in \Theta$.}
\]
Equality is attained by the (efficient) estimator $\xi = X$.
\end{teo}

In particular, if an estimator $\xi$ has finite $W^{\alpha,2}$ risk for some $\alpha \in [1/2, 1)$ and $u \in \Theta$, then it is not unbiased.

\begin{proof}
We introduce the notation $\Delta_t :=\xi_t-u_t$, for $t\in[0,T]$, so that, by Fubini theorem, we write 
\[
\E^u\left[\Vert \xi-u\Vert^2_{W^{\alpha,2}_0}\right]=\int_0^T\int_0^T \frac{\E^u[\,|\Delta_t-\Delta_s|^2\,]}{|t-s|^{2\,\alpha+1}}\,dt\,ds.
\]

If $\xi$ is an unbiased estimator and $v \in H^1_0$, we argue (once again) to obtain  \eqref{eq:identity-cramer-rao-bm}, and subtract such identity for $s$, $t \in [0,T]$, thus
\[
v(t)-v(s)=\E^u\Big[\,(\Delta_t-\Delta_s)\,\int_0^T\dot{v}(r)\,dX^u_r\,\Big].
\]
Hence, Cauchy-Schwarz inequality and Ito isometry give the lower bound
\[
\E^u[\,|\Delta_t-\Delta_s|^2\,]\geq \frac{|v(t)-v(s)|^2}{\int_0^T\dot{v}^2(s)\,ds},\quad  \text{for $s$, $t \in [0,T]$.}
\]
We let $\dot{v}(r)=1_{[s\wedge t,s\vee t]}(r)$, so that
\[
\E^u[\,|\Delta_t-\Delta_s|^2\,]\geq |t-s|\quad  \text{for $s$, $t \in [0,T]$.}
\]
The Cramer-Rao then follows:
\[
\int_0^T\int_0^T \frac{\E^u[\,|\Delta_t-\Delta_s|^2\,]}{|t-s|^{2\,\alpha+1}}\,dt\,ds\geq \int_0^T\int_0^T \frac{1}{|t-s|^{2\,\alpha}}\,dt\,ds.
\]
Finally, if $\xi = X$, then $X - u = X^u$, thus it holds
\[
E^u[\,|X^u_t-X^u_s|^2\,]=|t-s|, \quad \text{for $s$, $t \in [0,T]$.}
\]
and the Cramer-Rao lower bound is attained:
\[
\int_0^T\int_0^T \frac{\E^u[\,|X^u_t-X^u_s|^2\,]}{|t-s|^{2\,\alpha+1}}\,dt\,ds=\int_0^T\int_0^T \frac{1}{|t-s|^{2\,\alpha}}\,dt\,ds.
\]
\end{proof}

In the case of a general exponent $p \in (1,\infty)$ (with $q = p/(p-1)$), arguing similarly, we obtain the following bound, in $W^{\alpha,p}$. As above, we let $c_q = \E[ |Y|^q]$ be the $q$-th moment of a standard Gaussian random variable.
  
\begin{teo}[Cramer-Rao inequality in $W^{\alpha,p}$]  Let $\xi$ be an unbiased estimator. For every $\alpha\in(0,1)$, $p \in (1,\infty)$, it holds
\[
\E^u\left[\Vert \xi-u\Vert^p_{W^{\alpha,p}_0}\right]\geq \frac{1}{c_q^{p/q}}\,\frac{2\,T^{1-p\alpha+p/2}}{p \max\{0, (1/2-\alpha)\}\,(1+p(1/2-\alpha))}.
\] 
\end{teo}

Since
\[
E^u[\,|X^u_t-X^u_s|^p\,]=c_p\,|t-s|^{p/2},
\]
the risk of the estimator $\xi = X$ is given by
\[
\int_0^T\int_0^T \frac{\E^u[\,|X^u_t-X^u_s|^p\,]}{|t-s|^{p\,\alpha+1}}\,dt\,ds=c_p\,\int_0^T\int_0^T \frac{1}{|t-s|^{p\,\alpha+1-p/2}}\,dt\,ds.
\]
As in Remark~\ref{oss:lp} above, we conclude that $X$ is not an efficient estimator with respect to the risk in $W^{\alpha,p}$, for $p\neq 2$.

\begin{oss}
Before we conclude this section, we remark that all the bounds above can be generalized (at least) to the case of a continuous Gaussian martingale, with quadratic variation process $\int_0^t \sigma_s^2\,ds$, $t \in [0,T]$ and also by introducing different energies, such as
\[
\int_0^T\int_0^T\frac{|u(t)-u(s)|^p}{|t-s|^{\alpha\, p+1}}\,\mu(dt,ds),
\]
where $\mu$ is a measure on $[0,T]$ (a natural choice would be to take $\mu$ somehow related to $\sigma^2$). However, we choose to limit the discussion to the case of the Brownian motion, to limit technicalities and emphasize the role played by the norm chosen to estimate the risk. 
\end{oss}

\section{Intensity estimation for the Cox process}\label{sec:cox}
 
Throughout this section, we fix $T\geq 0$ and let $X=(X_t)_{t\in[0,T]}$ be a Poisson process defined on some filtered probability space $(\Omega,\F,(\F_t)_{t\in[0,T]},\p)$, with jump times $(T_k)_{k \ge 1}$ (for $k \ge 1$, we let $T_k(\omega) = T$ in the eventuality that no $k$-th jump occur). As a space of parameters $\Theta$, we consider the set of all absolutely continuous, (strictly) increasing, $\F_0$-measurable processes $u = (u_t)_{t \in [0,T]}$ such that their a.e.\ derivatives $(\dot{u}_t)_{t\in[0,T]}$ satisfy the assumptions of Girsanov theorem for the Poisson process (the proofs work also for slightly smaller sets). Given $u\in\Theta$, we define the probability $\p^u:=L^u\,\p,$ where
\[
L^u:=\prod_{k=1}^{X_T}\dot{u}_{T_k}\,\exp\Big(\,-\int_0^T(\,\dot{u}_s-1\,)\,ds\,\Big).
\]
Girsanov theorem entails that, with respect to the probability $\p^u$, the process $X$ is a Cox process with intensity $(\dot{u}_t)_{t\in [0,T]}$ (see e.g.\ \cite[Section 8.4]{MR2568861} for details on related doubly stochastic Poisson processes). Notice that $\p^u(A)$ does not depend on $u$ for $A \in \F_0$, thus e.g.\ for $t \in [0,T]$, $v \in \Theta$, $u_t$ is integrable with respect to $\p^v$ and its expectation $\E^v[u_t]$ actually does not depend on $v$.

We address the problem of estimating $u$, or equivalently the intensity of $X$ w.r.t.\ $\p^u$, based on a single observation of $X$. In the case of a deterministic intensity, i.e\ when $X$ is an inhomogeneous Poisson process, this is investigated e.g.\ in \cite{MR2486115}, and, similarly to the case of shifted Brownian motion, the following definition is given.

\begin{de}
Any measurable stochastic process $\xi:\Omega\times[0,T]\to \R$ is called an estimator of the intensity $u$. An estimator of the intensity $u$ is said to be unbiased if, for every $u\in\Theta$, $t\in[0,T]$, $\xi_t$ is integrable and it holds $\E^u[\,\xi_t\,]=\E[\,u_t\,]$.
\end{de} 

As in the previous section, we forgo to specify ``of the intensity $u$'' and simply refer to estimators.

Privault and R\'evelliac studied the estimation problem, in the case of deterministic intensities, w.r.t.\ the risk in $L^2(\mu)$, defined as in \eqref{eq:risk-l2}, for any finite Borel measure on $[0,T]$. Their set of parameters $\Theta$ consists of all the space of deterministic absolutely continuous, increasing processes $u$, see \cite[Definition 2.1]{MR2486115}. We briefly show how a similar argument indeed applies as well to the case of stochastic intensities.

\begin{teo}[Cramer-Rao inequality in $L^2(\mu)$] For any unbiased estimator $\xi$, it holds
\[
\E^u[\,\Vert \xi-u\Vert^2_{L^2(\mu)}\,] \geq \int_0^T \E^u[u_t]\,\mu(dt), \quad \text{for every $u\in \Theta$,}
\] 
and equality is attained by the (efficient) estimator $\xi=X$.
\end{teo}

\begin{proof}
For every process $v\in\Theta$, since $\xi$ is unbiased we have
\[
\E^{u+\epsilon\,v}[\,\xi_t\,] =\E^{u+\epsilon\,v}[\,u_t+\epsilon\,v_t\,]
 =\E^{u+\epsilon\,v}[\,u_t\,]+\epsilon\,\E^{u+\epsilon v}[v_t], \quad \text{ for $t\in[0,T]$.}
\]
Differentiating w.r.t.\ $\epsilon$, as in in \cite[Proposition 2.3]{MR2486115} we obtain the identity
\begin{equation}
\label{eqn:CR-cox}
\begin{aligned}
\E^u[\,v_t\,]&=\frac{d}{d\epsilon}\Big|_{\epsilon=0}\E^{u+\epsilon\,v}[\,\xi_t-u_t\,]\\ 
&= \E^u\Big[\,(\xi_t-u_t)\,\int_0^T\frac{\dot{v}_s}{\dot{u}_s}\,(dX_s-\dot{u}_s\,ds)\,\Big].
\end{aligned}
\end{equation}
By Cauchy-Schwarz inequality and the fact that $X$ is a Cox process with intensity $\dot{u}$, we get, for $t \in [0,T]$,
\[
\E^u[\,v_t\,]^2\leq \E^u[\,(\xi_t-u_t)^2\,]\,\E^u\Big[\,\int_0^T\frac{\dot{v}^2_s}{\dot{u}_s}\,ds\,\Big]\,\text{ thus }\,\E^u[\,(\xi_t-u_t)^2\,]\geq \E^u[\,u_t\,],
\]
once we let $\dot{v}= \dot{u}\, 1_{[0,t]}$. The thesis follows by integration w.r.t.\ $\mu$.
\end{proof}

Differently from the case of Brownian motion, the lower bound depends on the parameter $u \in \Theta$. This is quite natural in view of the classical, finite-dimensional, Cramer-Rao lower bound, where the inverse of the Fisher information appears, measuring the local regularity of the densities: when $u$ is small, the density becomes very peaked and the bound becomes trivial.

Since the intensity $u \in \Theta$ is absolutely continuous, also in this case we investigate lower bounds for the  $H^1_0$ risk: also in this case, no unbiased estimators exist. In the next result, we also collect the case of fractional Sobolev spaces $W^{\alpha,2}$, for $\alpha \in (0,1)$.

\begin{teo} For any unbiased estimator $\xi$, $\alpha \in (0,1)$, it holds
\[ \E^u[\, \| \xi - u \|_{W^{\alpha,2}_0}^2 \, ] \ge 2 \int_0^T \E^u[\dot{u}_r] \int_r^T\int_0^r \frac{1}{(t-s)^{2\alpha+1}} ds dt dr, \]
for every $u \in \Theta$. There exists no unbiased estimator $\xi$ with finite risk in $W^{\alpha,2}$ for $\alpha \in [1/2,1)$, as well as in $H^1_0$.
\end{teo}

\begin{proof}
We subtract \eqref{eqn:CR-cox} for two different times $s$, $t \in [0,T]$, and apply Cauchy-Schwarz, obtaining
\[
\E^u[\,|\Delta_t-\Delta_s|^2\,] \geq \frac{\E^u[\,|v_t-v_s|\,]^2}{\E^u\left[\int_0^T\frac{\dot{v}^2_s}{\dot{u}_s}\,ds\right]}.
\]
Hence, taking $\dot{v}_r=1_{[s\wedge t,s\vee t]}(r)\,\dot{u}_r,$ we have
\[
\E^u[\,|\Delta_t-\Delta_s|^2\,]\geq \E^u[\,|u_t-u_s|\,],\quad \text{for every $s$, $t \in [0,T]$.}
\]
If $s<t$, then the right hand side above coincides with $\E^u[\,\int_s^t \dot{u}_r \, dr \,]$. Integrating with respect to $s$, $t \in [0,T]$, with measure $|t-s|^{-2\alpha -1} dtds$, we obtain the required inequality. To deduce that no unbiased estimators with finite risk exist, it is sufficient to notice that the double integral equals $+\infty$, for $\alpha \in [1/2,1)$, and $\E[\dot{u}_r]>0$ for a.e.\ $r \in [0,T]$. The case of $H^1_0$ follows at once from inequality \eqref{eq:walpha-h1}.
\end{proof}

\section{Stochastic calculus of variations }\label{sec:malliavin}
In this section, we briefly recall some results concerning Malliavin Calculus on the classical Wiener space (we refer to the monograph \cite{MR2200233} for details), limiting ourselves the essentials for constructing super-efficient estimators. 

In the framework of Section \ref{sec:drift-bm}, i.e.\ if $X = (X_t)_{t \in [0,T]}$ is a Brownian motion (on the finite interval $[0,T]$), defined on some filtered probability space $(\Omega,\F,(\F_t)_{t \in [0,T]},\p)$, we introduce the space $\mathcal{S}$ of smooth functionals, as those in the form
\[
F=\phi\left( X_{t_1}, \dots, X_{t_n}\right),\]
for some $t_1, \ldots, t_n \in [0,T]$ and $\phi \in \mathcal{C}^\infty_b(\R^n)$ ($n\ge 0$). The Malliavin derivative $DF$ is then defined as the $L^2(0,T)$-valued random variable
\[
D_tF:=\sum_{i=1}^n\frac{\partial\phi}{\partial x_i}\left( X_{t_1}, \dots, X_{t_n}\right) 1_{[0,t_i]}(t),\, \text{ for a.e.\ $t\in[0,T]$.}
\]
For $h \in L^2(0,T)$, we let $D_{h} F:=\int_0^T D_tF\,h(t)dt$ (in the classical Wiener space framework, this corresponds to differentiation along the direction in $H^1_0$ given by $\tilde{h}(t) = \int_0^t h(s)ds$, $t \in [0,T]$: differently from the previous sections, we prefer to focus on the space $L^2(0,T)$ instead of $H^1_0$). The Cameron-Martin theorem entails the following integration by parts formula for smooth functionals.
\begin{prop}
Let $F\in\mathcal{S}$ and $h \in L^2(0,T)$. Then, it holds
\begin{equation}\label{eq:ibp-bm}
\E[\,D_{h}F\,]=\E\left[\,F  h^*\right],
\end{equation}
where we let $h^* = \int_0^T h(s) dX_s$ be the Ito(-Wiener) integral.
\end{prop}
A straightforward consequence of the integration by parts formula above is closability for the operator $D:\,\mathcal{S}\subset L^2(\Omega)\to L^2(\Omega\times[0,T])$. The domain of its closure defines the Sobolev-Malliavin space $\mathbb{D}^{1,2}$, on which the operator $D$ extends continuously.

\begin{prop}[chain rule]
Let $F_1, \ldots, F_n\in\mathbb{D}^{1,2}$ and $\phi\in C^1_b(\R^n)$. Then, it holds $\phi(F_1,\ldots, F_n)\in \mathbb{D}^{1,2}$ with 
\[
D_t\phi(F_1, \ldots, F_n)=\sum_{i=1}^n\frac{\partial\phi}{\partial x_i}\,(F_1, \ldots, F_n)\,D_tF_i, \quad \text{for a.e.\ $t \in [0,T]$.}
\] 
\end{prop} 

\begin{oss}[Malliavin Calculus for a Cox process]\label{rem:calculus-cox}
It seems reasonable to develop a theory of differential calculus for Cox processes, akin to that for Poisson processes introduced \cite{MR2486115}: in the setting of Section \ref{sec:cox}, i.e., if we let $(X_t)_{t \in [0,T]}$ be a Cox process on $(\Omega,\F,(\F_t)_{t \in [0,T]},\p)$, with intensity $\lambda = (\lambda_t)_{t \in [0,T]}$ and jump times $(T_k)_{k \ge 1}$. Then, we let $\mathcal{S}$ be the space of random variables $F$ in the form
\[
F=f_0\,1_{\{X_T=0\}}+\sum_{n=1}^\infty 1_{\{X_T=n\}}\,f_n(T_1,\ldots,T_n),
\]
where, for $n \ge 0$, $f_n : \Omega \times \R^n \to \R$ is bounded, measurable with respect to $\F_0 \times \mathcal{B}(\R^n)$ (i.e.\ its randomness depends only on $\lambda$) and for every $\omega \in \Omega$, $f_n(\omega; \cdot)$ is $\C^\infty_b(\R^n)$ and symmetric, i.e., $f_n(\omega; t_1, \ldots, t_n)$ is left unchanged by any permutation of the coordinates $(t_1, \ldots, t_n)$ and that, for every $n \ge 0$, it holds $f_n(\omega; t_1, \ldots, t_n)=f_{n+1}(\omega; t_1, \ldots, t_n, T)$, for $\omega \in \Omega$, $t_1, \ldots, t_n \in \R$.

 For $F \in \mathcal{S}$, we may let $DF(\omega) \in L^2(0,T)$
\[
D_tF:=-\sum_{n=1}^\infty 1_{\{X_T=n\}}\,\sum_{k=1}^n 1_{[0,T_k]}(t)\,\frac{1}{\lambda_{T_k}}\,\partial_k f_n(T_1,\ldots,T_n)\,\lambda_t,
\]
for a.e.\ $t\in[0,T]$. 

One can prove the validity of the chain rule and an integration-by-parts formula, providing some notion of divergence, thus defining Sobolev-Malliavin spaces in this setting. However, it is presently not clear how to effectively use such calculus to produce super-efficient Stein-type estimators, see Remark \ref{rem:stein-cox} below. 
\end{oss}

\section{Super-efficient estimators}\label{sec:stein}

In this section, we address the problem of Stein type, super-efficient estimators for the drift of a shifted Brownian motion, with respect to risks computed in the Sobolev spaces introduced above.

For $L^2(\mu)$-type risks, super-efficient estimators in the form $X+\xi$ were first studied in \cite{MR2458197}. Privault and R\'{e}veillac consider a process $\xi_t=D_{1_{[0,t]}}\log F$, $t \in [0,T]$, where $F$ is any $\p$-a.s.\ non-negative random variable in $\mathbb{D}^{1,2}$ such that $\sqrt{F}$ is $\Delta$-superharmonic w.r.t.\ a suitable ``Laplacian'' operator, actually related to the structure of the risk considered (which is not, in the Gaussian case, the usual Gross-Malliavin Laplacian). 
We show that a similar approach leads to super-efficient estimators also in fractional Sobolev spaces $W^{\alpha, 2}$, for $\alpha \in [0,1/2)$ (of course, this perturbative approach does not provide any information for larger values of $\alpha$). Indeed, for every $\xi = (\xi_t)_{t \in [0,T]}$, with $\E^u[ \|\xi \|^2_{W^{2,\alpha}_0} ]  < \infty$, we write
\begin{align*}
\E^u[\Vert X+\xi-u\Vert^2_{W^{\alpha,2}_0}]  & = \E^u \left[ \Vert X-u\Vert^2_{W^{\alpha,2}_0}+ \Vert \xi\Vert^2_{W^{\alpha,2}_0}\right] + \\
& +2\int \E^u\big[\,(\xi_t-\xi_s)\,[ (X_t-u_t)-(X_s-u_s)] \,\big] d\mu_\alpha(s,t),
\end{align*}
where we introduce the Borel measure $\mu_\alpha(ds,dt) = 2 \, (t-s)^{-2\alpha-1} 1_{\{s < t\}}dsdt$ on $[0,T]^2$. If $\xi_t - \xi_s \in \mathbb{D}^{1,2}$, for every $s$, $t \in [0,T]$, with $s < t$, the integration by parts \eqref{eq:ibp-bm}  for the Malliavin derivative (to be rigorous, we should write in what follows $D^u$, because the derivative is built with respect to the probability $\p^u$, not $\p$), entail
\begin{align*}
\E^u\big[\,(\xi_t-\xi_s)\,[(X_t-u_t)-(X_s-u_s)]\,\big]&=\E^u\big[\,(\xi_t-\xi_s)\,(X^u_t-X^u_s)\,\big]\\
&=\E^u\left[(\xi_t-\xi_s) \, 1_{[s,t]}^* \right]\\
&=\E^u\Big[\, \tilde{D}_{s,t} (\xi_t-\xi_s)\,\Big].
\end{align*}
where $\tilde{D}_{s,t}  F:= D_{1_{[s,t]}} \int_s^t D_rF\,dr$. Hence, if we let $\rho = \E^u [ \Vert X-u\Vert^2_{W^{\alpha,2}_0}]$ denote the Cramer-Rao lower bound, we deduce
\[
\E^u\left[\Vert X+\xi-u\Vert^2_{W^{\alpha,2}_0}\right]=\rho+\int\E^u\left[|\xi_t-\xi_s|^2+2\tilde{D}_{s,t}(\xi_t-\xi_s)\right]\mu_\alpha(ds,dt).
\]
It is then convenient to introduce the following notion of Laplacian,
\begin{equation}\label{eq:delta-alpha}
\Delta_\alpha F: =\int_{[0,T]^2} ( \tilde{D}_{s,t} ) ^2 F \mu_\alpha(ds,dt),
\end{equation}
initially defined on $\mathcal{S}$. Arguing e.g.\ as in \cite[Proposition 4.5]{MR2458197}, it is possible to show that $\Delta_\alpha: \mathcal{S} \subseteq L^2(\Omega, \p^u) \to L^2(\Omega, \p^u)$ is closable and that the random variables $G \in \mathbb{D}^{1,2}$, with
\begin{equation}\label{eq:useful-domain-delta-alpha} \text{$\tilde{D}_{s,t} G\in\mathbb{D}^{1,2}$, for a.e.\ $s$, $t\in[0,T]$  and $\tilde{D}_{s,t}^2 G\in L^2\left(\Omega \times [0,T]^2,\p \times \mu_\alpha\right)$,}\end{equation}
belong to the domain of the closure, so that $\Delta_\alpha G$ is well-defined (actually, by the same expression as in \eqref{eq:delta-alpha}). Moreover, the operator $\Delta_\alpha$ is of diffusion type, i.e., for every $F_1, \ldots, F_n \in \mathcal{S}$, $\phi \in C^2_b(\R^n)$, the function $\phi \circ \mathbf{F}$ (we write $\mathbf{F} = (F_1, \ldots, F_n)$) belongs to the domain of $\Delta_\alpha$, and it holds
\begin{equation} \label{eq:chain-rule-delta-alpha}\Delta_\alpha ( \phi \circ \mathbf{F})  = \sum_{i=1}^n \frac{\partial \phi}{\partial x_i} (\mathbf{F}) \,  \Delta_\alpha F_i +  \sum_{i,j=1}^n \frac{\partial^2 \phi}{\partial x_i \partial x_j}(\mathbf{F})  \, \Gamma_{\alpha}(F_i, F_j), \quad \text{$\p$-a.e.\ in $\Omega$,}\end{equation}
with $\Gamma_{\alpha}(F_i, F_j) = \int_{[0,T]^2} \tilde{D}_{s,t} F_i \tilde{D}_{s,t} F_j \mu_\alpha(ds,dt)$, for $i$, $j \in \{1, \ldots, n\}$ (the Malliavin matrix associated to $(F_i)_{i=1}^n$). This identity, by density, extends under natural integrability assumptions on $\mathbf{F}$ as well as on $\phi$.

The operator $\Delta_\alpha$ enters in the picture if we assume that process $\xi$ is of the form $\xi_t=\tilde{D}_{0,t}\log F^2$, $t \in [0,T]$, for some $\p$-a.e.\ positive random variable $F \in \mathbb{D}^{1,2}$, with $G = \log F^2$ satisfying \eqref{eq:useful-domain-delta-alpha}.
If we are in a position to apply the chain rule \eqref{eq:chain-rule-delta-alpha}, it holds 
\begin{align*}
\Delta_\alpha \log F^2  & = 2 \frac{ \Delta_\alpha F }{F} - \frac{2}{F^2}\Gamma_\alpha(F, F) \\
& = \frac{ 2 \Delta_\alpha F }{F} - \frac{1}{2} \Gamma_\alpha(\log F^2, \log F^2)
\end{align*} 
which can be explicitly written in terms of $\xi$ as
\[ \frac{ 4\Delta F }{F} = \int_{[0,T]^2} \left[2 \tilde{D}_{s,t} (\xi_t - \xi_s) + |\xi_t - \xi_s|^2 \right]\mu_\alpha(ds,dt).\]
As a result, we obtain
\[
\E^u\left[\Vert X+\xi-u\Vert^2_{W_0^{\alpha,2}}\right]=\rho+4\,\E^u\left[\frac{\Delta_\alpha F}{F}\right].
\]

Therefore, in order to find super-efficient estimators, it is enough to prove existence of some $\xi$ (independent of $u$) that can be written in terms of some $F$ (possibly depending on $u$), with $\Delta_\alpha F \le 0$ (i.e., super-harmonic) with strict inequality on a set of positive $\p^u$ (or equivalently $\p$) measure. In case of shifted Brownian motion, we provide the following

\begin{ex}\label{ex:stein-bm}
Let $F$ be a r.v.\ of the form $F=\phi\big(X_{t_1}, X_{t_2}-X_{t_1}, \ldots, X_{t_n}-X_{t_{n-1}})$, for some $0 =t_0 < t_1 < \ldots < t_n \le T$ (with $\phi: \R^n \to \R^n$ sufficiently regular, in order to perform all the computations below). Then, by \eqref{eq:chain-rule-delta-alpha}, we can express $\Delta_\alpha F$ in terms of $\nabla \phi$, $\nabla^2 \phi$, $\Delta_\alpha (\delta_i X)$ and 
\[ \Gamma_\alpha( \delta_i X, \delta_j X ) = \int_{[0,T]^2} \tilde{D}_{s,t} \delta_i X \tilde{D}_{s,t} \delta_j X \mu_\alpha(ds,dt), \quad \text{for $i$, $j \in \{1, \ldots, n\}$,}\]
with the notation $\delta_i X = X_{t_{i}} -X_{t_{i-1}}$.

Before we proceed further, we have to take into account that, with different probabilities $\p^u$, the r.v.'s may have different derivatives $D F= D^u F$ and Laplacians $\Delta_\alpha F = \Delta^u_\alpha F$, since the calculus w.r.t.\ $\p^u$ is ``modelled'' on the process $X^u= X-u$, thus, for $h \in L^2(0,T)$, $t \in [0,T]$, it holds
\[ D_h X_t = D_h X_t^u + D_h u_t = \int_0^t h(s) ds + D_h u_t\]
and
\[\Delta_\alpha X_t = \Delta_\alpha X_t^u + \Delta_\alpha u_t = \Delta_\alpha u_t,\]
provided that $u_t$ is sufficiently regular. To proceed further with computations, we assume that the process $u$ is deterministic i.e.\ we restrict the space of parameters $\Theta$ to $H^1_0$ only, so that $D_h u_t = \Delta_\alpha u_t = 0$, ruling out the problem of possible dependence upon $u$ of the Malliavin calculus that we consider.  Then, \eqref{eq:chain-rule-delta-alpha} reduces to
\[ \Delta_\alpha F= \sum_{i,j=1}^n \frac{\partial^2 \phi}{\partial x_i \partial x_j} a_{i,j}, \]
where, for $i$, $j \in \{1, \ldots, n\}$, with $t_0 = 0$, 
\[
a_{i,j}:=\int_{[0,T]^2} \int_s^t 1_{[t_{i-1}, t_{i}]}(r)dr  \int_s^t 1_{[t_{j-1}, t_{j}]}(r)dr\mu_\alpha(dt,ds).
\]

To prove that the symmetric matrix $A := (a_{ij})_{i,j=1}^n$ is well-defined and invertible, we argue as follows: for every $v = (v_i)_{i=1}^n$, it holds, using the notation $\ang{Av, v} := \sum_{i,j}^n  a_{i,j} v_i v_j$,
\begin{align*}
\ang{Av, v} & =  \int_{[0,T]^2} \sum_{i,j}^n  v_i v_j \int_s^t  1_{[t_{i-1}, t_{i}]}(r)dr  \int_s^t 1_{[t_{j-1}, t_{j}]}(r)dr\mu_\alpha(dt,ds)\\
& = \int_{[0,T]^2} \left( \int_s^t \sum_{i=1}^n v_i 1_{[t_{i-1}, t_{i}]}(r)dr \right)^2\mu_\alpha(dt,ds)\\
& = \int_{[0,T]^2} |\tilde{v}(t) - \tilde{v}(s)|^2 \mu_\alpha(dt,ds) = \|\tilde{v} \|_{W^{\alpha, 2}_0}^2,
\end{align*}
where we let $\tilde{v}(t) = \int_0^t \sum_{i=1}^n 1_{[t_{i-1}, t_{i}]}(s) v_i ds$. From this identity and \eqref{eq:walpha-h1} we deduce that $A$ is well-defined, while non-degeneracy follows from the fact that, if $\|\tilde{v} \|_{W^{\alpha, 2}_0} =0$, then $\tilde{v}$ is constant, which cannot happen except when $v = 0$.

We let $B:=(b_{i,j})_{i,j=1}^n$ be the inverse matrix of $A$, and consider the function
\[
\phi(x):=  \ang{Bx,x}^a, \quad \text{$x \in \R^n$,}
\]
for a suitable choice of $a \in \R$. Then, by formally applying the the chain rule in $\R^n$, it holds 
\[
\sum_{i,j}^n \frac{\partial^2 \phi}{\partial x^i \partial x^j} a_{i,j} = 2 a ( 2(a-1)+n)\ang{Bx,x }^{a-1},
\]
which suggests the choice $a\in(1-n/2, 0)$ (and $n \ge 3$). However, for $a$ in this range, $\phi$ is not $C^2_b(\R^n)$ and in order to rigorously conclude super-efficiency for an estimator in the form $X_t+\tilde{D}_{0,t} \log F^2$, $t \in [0,T]$, we have to justify all the applications of the chain rule above. Indeed, the only 
 non-trivial step is to prove the following estimate, for every $u \in H^1_0$:
\[ \E^u\left[ \ang{B (\delta X) , (\delta X) }^{-1} \right] < \infty. \]
In turn, this holds true because we may pass to the joint law of $\delta X = (\delta_i X)_{i=1}^n$, which is Gaussian non-degenerate (possibly non-centred) and the integrand can then be estimated from above by some constant times the function $x \mapsto |x|^{-2}$ (here the assumption $n \ge 3$ plays a role too).

Next, to prove e.g.\ that $\log F^2 \in\mathbb{D}^{1,2}$, with
\[ D_t \log F^2 = 2 a \frac{\sum_{i,j=1}^n b_{i,j} \delta_i X  1_{ [t_{j-1}, t_{j}] } (t) }{ \ang{B (\delta X), (\delta X)} }, \quad \text{for a.e.\ $t \in [0,T]$,}\]
it is sufficient to notice that, assuming this identity true, then we could estimate, by Cauchy-Schwarz inequality,
\[ \int_0^T \E^u [ | D_t \log F^2 | ^2] dt \le  4a^2 T \operatorname{trace}(B) \E^u \left[\ang{ B (\delta X), (\delta X) }^{-1} \right]. \]
This a priori estimate entails $\log F^2 \in\mathbb{D}^{1,2}$, by suitably approximating the function $z \mapsto \log z$ with smooth functions.

Similarly, to estimate $\E [ \| \xi \|_{W^{\alpha,2}_0}^2 ]$, we apply Cauchy-Schwarz and deduce, for $s$, $t \in [0,T]$, with $s< t$,
\[ \E^u [ | \tilde{D}_{s,t} \log F^2|^2] \le 4a^2 (t-s) \operatorname{trace}(B) \E^u \left[\ang{ B (\delta X), (\delta X) }^{-1}\right],\]
which can be integrated with respect to $\mu_{\alpha}$ (recall that $\alpha \in (0,1/2)$).
\end{ex}

In conclusion, the example above shows that, in the case of deterministic shifts, i.e., $\Theta = H^1_0$, we are able to explicitly build super-efficient Stein-type estimators. Although it seems reasonable, we do not know whether this technique can be extended to stochastic shifts; it would be even more interesting to provide super-efficient \emph{adapted} estimators, see also Remark \ref{oss:adapted} above.

\begin{oss}[Stein estimators for Cox processes]\label{rem:stein-cox}
In case of Cox processes, nothing prevents us from performing similar argument using, in place of Malliavin calculus, the calculus sketched in Remark \ref{rem:calculus-cox}. The case of Poisson processes and $L^2(\mu)$-type risks is investigated in \cite{MR2486115}. However, here we currently face  a strong limitation to provide explicit examples, due to the possible dependence upon $u$ (i.e., $\lambda$) of the Malliavin calculus. Let us remark that a similar limitation is also present in \cite{MR2486115} and perhaps, at least in the one-dimensional parametric cases considered in \cite[Section 5]{MR2486115}, one might similarly provide explicit examples of super-efficient estimators also with respect to Sobolev risks, but the general, infinite-dimensional parametric problem would still be open.
\end{oss}

\printbibliography
\end{document}